\documentclass[11pt]{amsart}
\usepackage[T1]{fontenc}
\usepackage{lmodern}
\usepackage[a4paper,margin=29mm]{geometry}
\usepackage{amsmath,amssymb,mathtools,microtype}
\usepackage{enumitem,needspace}
\usepackage{xcolor}
\usepackage[colorlinks=true,linkcolor=blue!45!black,citecolor=blue!45!black,urlcolor=blue!45!black]{hyperref}
\hypersetup{pdftitle={On the identity component of a strongly graded ring - with applications to graph algebras},pdfsubject={Universal module constructions and matrix presentations of identity components},pdfauthor={Raimund Preusser}}
\newtheorem{theorem}{Theorem}[section]

\theoremstyle{definition}

\newtheorem*{remark*}{Remark}

\newcommand{\N}{\mathbb N}
\newcommand{\Z}{\mathbb Z}

\newcommand{\id}{\operatorname{id}}
\newcommand{\End}{\operatorname{End}}
\newcommand{\diag}{\operatorname{diag}}

\newcommand{\Mat}{\mathbb M}

\newcommand{\Mf}{\mathfrak{M}}
\numberwithin{equation}{section}
\setlist[enumerate]{label=\textup{(\roman*)},leftmargin=*,itemsep=3pt}
\title[On the identity component of a strongly graded ring]{On the identity component of a strongly graded ring - with applications to graph algebras}
\author{Raimund Preusser}
\address{School of Mathematics and Statistics, Nanjing University of Information Science and Technology, Nanjing, China}  
\email{raimund.preusser@gmx.de}
\date{}
\begin{document}

\begin{abstract}
Let $G$ be a group, $K$ a commutative ring and $A$ a strongly $G$-graded $K$-algebra with a given homogeneous presentation. We show that the identity component $A_e$ can be obtained by successively applying some of Bergman's universal ring constructions to the underlying ring $K$. This yields a presentation of $A_e$ in terms of generators and relations, which depends on the given presentation of $A$ and the choice of ``frames'', which identify each $A_g$ with a direct summand of a free $A_e$ module of finite rank. We use this result to describe the $0$-components of canonically $\Z$-graded Leavitt algebras and Leavitt path algebras of finite hypergraphs without sinks and sources.
\end{abstract}
\maketitle

\section{Introduction}
Let $G$ be a group with identity $e$, $K$ a commutative ring and $A$ a strongly $G$-graded $K$-algebra with homogeneous presentation
\[A=\langle\mathcal X\mid \mathcal R\rangle.\]
Then each component $A_g$, viewed as a right $A_e$-module,
is isomorphic to a direct summand of a free $A_e$-module of finite
rank $d_g$ (see \S\ref{Subsection 3.1}). Moreover, each generator
$x\in\mathcal X$ defines $A_e$-module homomorphisms
\[f^x_{g}:A_g\to A_{\delta(x)g}\]
by left multiplication, where $\delta(x)$ denotes the homogeneous degree of $x$.

In \S \ref{Subsection 3.2}, we define a $K$-algebra $R$ using some of G. Bergman's universal ring constructions in \cite{Bergman} as follows. Starting from $K$, adjoin for every $g\ne e$ a universal
decomposition of the free module of rank $d_g$ into two direct summands.
Over the resulting algebra $B$, write these decompositions as
\[
 B^{d_g}=P_g\oplus P'_g.\]
Set $P_e=B$ and $P'_e=0$. Next, adjoin to $B$ universal module homomorphisms
\[
 \phi^{x}_{g}:P_g\longrightarrow P_{\delta(x)g}
\]
where $x\in\mathcal X$ and $g\in G$. Finally impose relations reflecting the relations $\mathcal R$, together with relations specifying the chosen projective coordinates. This can be realised using Bergman's construction of imposing relations. Bergman's proofs
provide a presentation of $R$, which we
record in \S\ref{subsec:matrix-presentation}.

The $A_e$-modules $A_g$ and $A_e$-module homomorphisms $f^{x}_{g}$ determine an algebra homomorphism $R\to A_e$. In \S\ref{Subsection 3.4} we show that this homomorphism has an inverse. Thus $R$ is isomorphic to the identity component $A_e$.

In the 1950's and 60's, W. Leavitt studied algebras that are now denoted by $L_K(m,n)$ and known as Leavitt algebras \cite{vitt56, vitt57, vitt62, vitt65}. For positive integers $m\leq n$, the algebra $L=L_K(m,n)$ can be described as the $K$-algebra $L$ with a universal $L$-module isomorphism $L^m\cong L^n$.

Leavitt path algebras of directed graphs were introduced by Abrams and Aranda Pino in 2005 \cite{AA}, and independently by Ara, Moreno and Pardo in 2007 \cite{AMP}. The definition of a Leavitt path algebra was inspired on the one hand by Leavitt's algebras $L_K(1,n)$, and on the other hand by the graph $C^*$-algebras. 

The usual Leavitt path algebras generalise Leavitt's algebras $L_K(1,n)$. There have been several attempts to introduce a generalisation of Leavitt path algebras which covers the algebras $L_K(m,n)$ with $m>1$ as well.
In 2012, Ara and Goodearl \cite{aragoodearl} introduced Leavitt
path algebras of separated graphs. These extend the usual Leavitt
path algebras and realise every Leavitt algebra $L_K(m,n)$ as a
full corner. In 2013, Hazrat \cite{H-1} introduced weighted
Leavitt path algebras, a class containing both the usual Leavitt
path algebras and all Leavitt algebras $L_K(m,n)$.

In 2020, the author of the current paper introduced Leavitt path algebras of hypergraphs \cite{Hyper}. These algebras simultaneously generalise the Leavitt path algebras of separated graphs and vertex-weighted graphs. The Leavitt path algebras of vertex-weighted graphs form a subclass of the class of weighted Leavitt path algebras, which is still large enough to embrace the usual Leavitt path algebras as well as all algebras $L_K(m,n)$. Abrams and Hazrat have demonstrated that Leavitt path algebras of vertex-weighted graphs are closely connected to sandpile models and chip-firing games \cite{AH}.

In Sections~\ref{sec:leavitt} and~\ref{sec:hyper}, we apply the main result of Section~\ref{sec:general} to canonically $\Z$-graded Leavitt algebras and Leavitt path algebras of hypergraphs. The hypergraphs are finite and have no sinks or sources. Under these assumptions their Leavitt path algebras are unital and strongly graded. 

Section~\ref{sec:leavitt} can be considered as ``warm-up'' for Section \ref{sec:hyper}, which contains a more general result. We prove that $R\cong L_K(m,n)_0$ is isomorphic to the coproduct $S(m,n)*_K S(n,m)$,
where $S(m,n)$ is the algebra generated by the entries of $n^p\times n^p$ matrices 
\[E_p^{k\ell}~(p\geq1;~1\leq k,\ell\leq m)\]
subject to the entrywise relations
\begin{equation*}
 E_p^{k\ell}E_p^{ab}=\delta_{\ell a}E_p^{kb},\qquad
 \sum_kE_p^{kk}=\diag(\underbrace{E_{p-1}^{11},\dots,E_{p-1}^{11}}_{n\text{ terms}})
\end{equation*}
where $E_0^{11}=(1)$ (the algebra $S(n,m)$ is defined analogously).

In Section \ref{sec:hyper}, we obtain a similar description of the $0$-component of a Leavitt path algebra of a finite hypergraph without sinks and sources. In this case, vertex supports must also be recorded, and the coproduct is amalgamated over the vertex algebra. Repeated
occurrences of a vertex in a source or range require separate coordinates. These adaptations affect the frame choices and the elimination of transition matrices.

\section{Preliminaries}\label{sec:prelim}
$\N$ denotes the set of positive integers, $\N_0$ the set of nonnegative integers, and $\Z$ the set of all integers.
\subsection{\texorpdfstring{Rings, algebras and $A$-algebras}{Algebras and A-algebras}}
All rings are associative and unital. Ring homomorphisms are unital, and by a module over a ring we mean a unital right module.

Throughout this paper $K$ denotes a commutative ring. An \emph{algebra} is a ring $A$ equipped
with a specified ring homomorphism $K\to\operatorname{Center}(A)$.
An \emph{algebra homomorphism} is a ring homomorphism commuting with the specified maps from $K$. 

Let $A$ be an algebra. An \textit{$A$-algebra} is an algebra $B$ with a
specified algebra homomorphism $A\to B$. The image of this homomorphism is not (!) required to be contained in $\operatorname{Center}(B)$. An \textit{$A$-algebra homomorphism} is an algebra homomorphism commuting with the specified maps from $A$. Coproducts
over $A$ are taken in the category of $A$-algebras.

\subsection{\texorpdfstring{$G$-graded algebras}{G-graded algebras}}
Let $G$ be a group with identity $e$. A \emph{$G$-graded algebra}
is an algebra $A$ together with a direct sum decomposition
\[
 A=\bigoplus_{g\in G}A_g,
\]
where each $A_g$ is an additive subgroup, $A_gA_h\subseteq A_{gh}$
for all $g,h\in G$, and the image of $K$ lies in $A_e$.
Thus every $A_g$ is a $K$-module. Here $A_gA_h$ denotes the
additive subgroup consisting of all finite sums of products $ab$ with
$a\in A_g$ and $b\in A_h$. The grading is \emph{strong} if
$A_gA_h=A_{gh}$ for all $g,h\in G$.

For $g\in G$, the subgroup $A_g$ is called the \textit{$g$-component} of $A$. We might call $A_e$ the \textit{identity component}, or if $G=\Z$, the \textit{zero component}. The elements of $\bigcup_{g\in G}A_g$ are called \textit{homogeneous}. For a homogeneous element $a\neq 0$, the unique $g\in G$ such that $a\in A_g$ is called the \textit{degree} of $a$, denoted by
$\delta(a)$. 


\subsection{\texorpdfstring{The monoid $\Mf_n(R)$}{The monoid M(n,R)}}
Let $R$ be a ring. If $m,n\in\N$, then $\Mat_{m\times n}(R)$ denotes the set of all $m\times n$ matrices with entries in $R$. If $A\in \Mat_{m\times n}(R)$ and $t\in \N$, then $\diag_t(A)\in\Mat_{tm\times tn}(R)$ denotes the block-diagonal matrix with $t$ copies of $A$ on the diagonal and zeroes elsewhere.


For fixed $n\in\N$ set
\[\Mf=\Mf_n(R):=\bigcup_{i,j\in \N_0} \Mat_{n^i\times n^j}(R).\] 
We define a multiplication $\star$ on $\Mf$ as follows (cf. \cite[Definition 7.1.9]{Survey}). Let $A,B\in \Mf$. Then there are $i,j,k,l\in\N_0$ such that $A\in \Mat_{n^i\times n^j}(R)$ and $B\in \Mat_{n^k\times n^l}(R)$.
The product $A\star B$ is defined by
\[A\star B=\begin{cases}
\diag_{n^{k-j}} (A)B,\quad&\text{ if }j\leq k,\\[5pt]
A\diag_{n^{j-k}} (B),\quad&\text{ if }j\geq k.
\end{cases}\]
Hence $A\star B$ is the usual matrix product of $A$ and $B$ if $j=k$. With a little effort one can show that $(\Mf,\star)$ is a monoid whose identity element is $(1)\in \Mat_{1\times 1}(R)$. 
For any $A\in \Mf$ we set 
\[A^{\star0}=(1)\quad\text{ and }\quad A^{\star p}:=\underbrace{A\star A\star\dots\star A}_{p\text{ factors}}\quad(p\geq 1).\]

\section{The identity component of a strongly graded algebra}
\label{sec:general}

\subsection{Basic observations}\label{Subsection 3.1}
Let $G$ be a group with identity $e$. Let $A$ be a
strongly $G$-graded algebra with a homogeneous presentation
\[A=\langle\mathcal X\mid \mathcal R\rangle,\]
where $\mathcal X$ is a set of symbols with assigned degrees
$\delta(x)\in G$, and $\mathcal R$ is a set of nonzero homogeneous
polynomials in $K\langle\mathcal X\rangle$. Since the grading is strong, we can choose for every $g\in G$ vectors 
\begin{equation*}
 U_g=(u_{g,1},\dots,u_{g,d_g})\in\Mat_{1\times d_g}(A_g)\quad \text{ and }\quad V_g=(v_{g,1},\dots,v_{g,d_g})^T\in\Mat_{d_g\times 1}(A_{g^{-1}})
\end{equation*}
such that $U_gV_g=1$. For $g=e$ we choose $U_e=V_e=(1)$.
We call the vectors $U_g, V_g~(g\in G)$ \textit{frames}.

Each matrix $V_gU_g\in\Mat_{d_g\times d_g}(A_e)$ is idempotent, and the maps
\begin{alignat*}{2}
 A_g &\longrightarrow (V_gU_g)A_e^{d_g},
 &\qquad\qquad (V_gU_g)A_e^{d_g} &\longrightarrow A_g,\\
 a &\longmapsto V_ga,
 &\qquad b &\longmapsto U_gb,
\end{alignat*}
are inverse $A_e$-module isomorphisms. It follows that $A_g$ is isomorphic to a direct summand of $A_e^{d_g}$ (since $A_e^{d_g}=(V_gU_g)A_e^{d_g}\oplus (I_{d_g}-V_gU_g)A_e^{d_g}$). Moreover, each generator $x\in\mathcal X$ defines a family of $A_e$-module homomorphisms 
\begin{align*}
f^x_{g}:A_g&\longrightarrow A_{\delta(x)g}\\
a&\longmapsto xa.
\end{align*}

\subsection{\texorpdfstring{The universal algebra $R$}{The universal algebra R}}\label{Subsection 3.2}
We use some of Bergman's universal ring constructions in \cite{Bergman} to define an algebra $R$. Starting from $K$, adjoin for every $g\ne e$ a universal
decomposition of the free module of rank $d_g$ into two direct summands. Such a decomposition can be obtained by adjoining an
endomorphism of the free module and imposing its idempotency, see \cite[p. 39]{Bergman}.
Over the resulting algebra $B$, write these decompositions as
\[
 B^{d_g}=P_g\oplus P'_g\]
and set $P_e=B$ and $P'_e=0$. 

Next, adjoin to $B$ universal module homomorphisms
\[
 \phi^x_g:P_g\longrightarrow P_{\delta(x)g}
 \qquad(x\in\mathcal X,\ g\in G).
\]
The modules and maps at each subsequent
stage are understood after extension of scalars to the coefficient
algebra at that stage. Let $C$ denote the resulting algebra.

For any nonempty word $w=x_1\cdots x_t$ over $\mathcal X$ and $g\in G$, we define the module homomorphism
\begin{equation*}
 \phi^w_g=
 \phi^{x_1}_{\delta(x_2\cdots x_t)g}\circ\cdots\circ
 \phi^{x_{t-1}}_{\delta(x_t)g}\circ \phi^{x_t}_g:P_g\to P_{\delta(w)g}.
\end{equation*}
For the empty word, set $\phi^1_g=\id_{P_g}$.
Extend these definitions $K$-linearly to obtain for every $g,h\in G$ and $a\in K\langle \mathcal X\rangle_h$ a module homomorphism
$\phi^a_{g,hg}:P_g\to P_{hg}$. The maps $\phi^{0}_{g,hg}:P_g\to P_{hg}$ are the zero maps. If $a$ is a nonzero homogeneous element of $K\langle \mathcal X\rangle$, then we might abbreviate $\phi^{a}_{g,\delta(a)g}$ to $\phi^a_g$. 

Choose for every $g\in G$ and $1\leq i \leq d_g$ a representative $\tilde u_{g,i}\in K\langle \mathcal X\rangle_g$ of $u_{g,i}$, and a representative $\tilde v_{g,i}\in K\langle \mathcal X\rangle_{g^{-1}}$ of $v_{g,i}$. We define the algebra $R$ as the universal quotient of $C$ over which
\begin{equation}\label{eq:general-module-relations}
 \phi^{\rho}_{g}=0,\qquad
 \phi^{\tilde u_{g,i}}_{e,g}=\pi_g\varepsilon_i,\qquad
 \phi^{\tilde v_{g,i}}_{g,e}=\varepsilon_i^{\mathsf T}\iota_g
\end{equation}
for all $\rho\in\mathcal R$, $g\in G$ and $1\leq i\leq d_g$.
Here $\varepsilon_i$ denotes the $i$-th standard basis column, also regarded
as coordinate inclusion $C\to C^{d_g}$, and
$\varepsilon_i^{\mathsf T}$ is the corresponding coordinate projection
$C^{d_g}\to C$. Moreover,
\[\iota_g: P_g\to C^{d_g}\quad\text{and}\quad \pi_g: C^{d_g}\to P_g\]
are the inclusion and projection maps, respectively. 
Every equation in (\ref{eq:general-module-relations}) is an equality
of maps between finitely generated projective modules, hence one can obtain $R$ from $C$ using Bergman's imposing-relations construction.

\subsection{\texorpdfstring{A presentation of $R$}{A matrix presentation of R}}
\label{subsec:matrix-presentation}
For every $g\in G\setminus\{e\}$, let $Q_g$ be a $d_g\times d_g$ matrix whose entries are symbols, and set $Q_e=(1)$. Moreover, let for every $x\in\mathcal X$ and $g\in G$, $T^{x}_g$ be a $d_{\delta(x)g}\times d_g$ matrix whose entries are symbols. Let $D$ denote the free algebra generated by the entries of the matrices $Q_g~(g\in G\setminus\{e\})$ and $T^x_g~(x\in\mathcal X,~ g\in G)$.

For a nonempty word $w=x_1\cdots x_t$ over $\mathcal X$ and $g\in G$ we define the matrix
\begin{equation*}\label{eq:general-word-evaluation}
 T^w_g=
 T^{x_1}_{\delta(x_2\cdots x_t)g}\cdots
 T^{x_{t-1}}_{\delta(x_t)g}T^{x_t}_g
 \in \Mat_{d_{\delta(w)g}\times d_g}(D).
\end{equation*}
For the empty word, set $T^1_g=Q_g$. Extend these definitions $K$-linearly to obtain for every $g,h\in G$ and $a\in K\langle \mathcal X\rangle_h$ a matrix 
$T^a_{g,hg}\in\Mat_{d_{hg}\times d_g}(D)$. The matrices $T^0_{g,hg}\in\Mat_{d_{hg}\times d_g}(D)$ are the zero matrices. If $a$ is a nonzero homogeneous element of $K\langle \mathcal X\rangle$, then we might abbreviate $T^{a}_{g,\delta(a)g}$ to $T^a_g$. 

It follows from the proofs of \cite[Theorems~3.1 and~3.2]{Bergman}
that $R$ is the quotient of $D$ by the entrywise relations below.
\begin{align}
 Q_g^2&=Q_g&&(g\in G\setminus\{e\}), \label{eq:idempotent}\\
 T^x_g&=Q_{\delta(x)g}T^x_gQ_g&&(x\in\mathcal X,~g\in G),
                         \label{eq:general-support}\\
 T^{\rho}_{g}&=0&&(\rho\in\mathcal R,~g\in G),
                         \label{eq:general-lifted-relations}\\
  T^{\tilde u_{g,i}}_{e,g}&=Q_g\varepsilon_i&&(g\in G,~1\leq i \leq d_g),\label{eq:general-normalization1}\\
 T^{\tilde v_{g,i}}_{g,e}&=\varepsilon_i^{\mathsf T}Q_g&&(g\in G,~1\leq i \leq d_g).
                         \label{eq:general-normalization2}
\end{align}

Moreover,
\begin{equation}
 P_g=Q_gR^{d_g},\qquad P'_g=(I_{d_g}-Q_g)R^{d_g},
 \qquad \phi^{x}_g(z)=T^{x}_gz\label{eq:extra}
\end{equation}
for all $g\in G$, $x\in \mathcal X$ and $z\in P_g$.
\subsection{\texorpdfstring{$R$ is isomorphic to $A_e$}{The identity component is isomorphic to R}}\label{Subsection 3.4}
\begin{theorem}
\label{thm:homogeneous-presentation}
The algebra $R$ defined above is isomorphic to the identity component $A_e$.
\end{theorem}
\begin{proof}
\emph{The homomorphism $\alpha$.}
Let $D$ be the algebra defined in \S \ref{subsec:matrix-presentation}.
Define an algebra homomorphism $\alpha_D:D\to A_e$ entrywise by
\[
 \alpha_D(Q_g)=V_gU_g,\qquad
 \alpha_D(T^x_g)=V_{\delta(x)g}xU_g.
\]
We leave it to the reader to check that $\alpha_D$ preserves the relations (\ref{eq:idempotent})-(\ref{eq:general-normalization2}) and hence induces an algebra homomorphism $\alpha:R\to A_e$.

\emph{The homomorphism $\beta$.}
Define the $R$-module
\[
 M=\bigoplus_{g\in G}P_g.
\]
For each $x\in\mathcal X$, define $\phi^x\in\End_R(M)$ by
$\phi^x|_{P_g}=\phi^{x}_g$. Let $F=K\langle \mathcal X\rangle$ and define the algebra homomorphism 
\[
 \beta_F:F\to\End_R(M),\qquad x\longmapsto\phi^x.
\]
We leave it to the reader to check that $\beta_F(\rho)=0$ for all $\rho\in \mathcal R$.
Hence $\beta_F$ induces an algebra homomorphism $\beta_A:A\to \End_R(M)$. If $a\in A_e$, then $\beta_{A}(a)$ clearly restricts to an endomorphism of $P_e=R_R$. Since $\End(R_R)\cong R$, we obtain the algebra homomorphism
\[
 \beta:A_e\to R,\qquad a\longmapsto\beta_A(a)(1).
\]

\emph{The composite $\beta\alpha$.}
It follows from (\ref{eq:general-normalization1}), (\ref{eq:general-normalization2}) and (\ref{eq:extra}) that
\begin{equation}\label{eq:new}
 \phi^{\tilde u_{g,j}}_{e,g}(1)=T^{\tilde u_{g,j}}_{e,g}=Q_g\varepsilon_j,\qquad
 \phi^{\tilde v_{g,i}}_{g,e}(z)= T^{\tilde v_{g,i}}_{g,e}z=\varepsilon_i^{\mathsf T}Q_gz
 \quad(z\in P_g).
\end{equation}
Clearly $\alpha((Q_g)_{ij})=v_{g,i}u_{g,j}$. It follows from (\ref{eq:new}) that
\[
 \beta\alpha((Q_g)_{ij})
 =\phi^{\tilde v_{g,i}}_{g,e}\bigl(\phi^{\tilde u_{g,j}}_{e,g}(1)\bigr)=\varepsilon_i^{\mathsf T}Q_gQ_g\varepsilon_j   =(Q_g)_{ij}.
\]
Similarly, $\alpha((T^{x}_{g})_{ij})=v_{h,i}xu_{g,j}$ where $h=\delta(x)g$. It follows from  (\ref{eq:new}) that
\[
 \beta\alpha((T^{x}_{g})_{ij})
 =\phi^{\tilde v_{h,i}}_{h,e}
     \bigl(\phi^{x}_{g}(\phi^{\tilde u_{g,j}}_{e,g}(1))\bigr)
 =\varepsilon_i^{\mathsf T}Q_hT^{x}_{g}Q_g\varepsilon_j
   =(T^{x}_{g})_{ij}.
\]
Since the entries of the matrices $Q_g$ and 
$T^{x}_{g}$ generate $R$, we conclude that $\beta\alpha=\id_R$.

\emph{The composite $\alpha\beta$.}
Let $a\in A_e$, and let $\tilde a\in F_e$ be a representative of $a$. Then $\beta(a)=T^{\tilde a}_{e,e}$ by the definition of $\beta$ and (\ref{eq:extra}).
It follows that 
\[\alpha(\beta(a))=\alpha(T^{\tilde a}_{e,e})=V_eaU_e=a.\]
Hence $\alpha\beta=\id_{A_e}$,
and thus $\alpha$ and $\beta$ are inverse to each other.
\end{proof}

\section{The zero component of a Leavitt algebra}\label{sec:leavitt}
\subsection{\texorpdfstring{The Leavitt algebra $L(m,n)$}{The Leavitt algebra L(m,n)}}
Fix positive integers $m\leq n$. The algebra
\begin{equation}\label{eq:Leavittstandard}
 L=L_K(m,n)=K\Big\langle x_{ij},y_{ji}~(1\leq i\leq m,
 1\leq j\leq n)\mid XY=I_m,\ YX=I_n\Big\rangle,
\end{equation}
where $X=(x_{ij})$ and $Y=(y_{ji})$, is called \textit{Leavitt algebra}. We give $L_K(m,n)$ its \textit{canonical $\Z$-grading} defined by
$\delta(x_{ij})=1$ and $\delta(y_{ji})=-1$ for all $i$ and $j$. Clearly 
\begin{equation}\label{stronglygraded}
X_{1\bullet}Y_{\bullet1}=Y_{1\bullet}X_{\bullet1}=1,
\end{equation}
where for a matrix $A$, $A_{1\bullet}$ denotes its first row and $A_{\bullet 1}$ its first column. It follows that $L_K(m,n)$ is strongly graded, see \cite[last paragraph before Example 1.1.16]{Hazrat}.

\subsection{Choosing the frames}
For $p\geq1$, choose the frames
\begin{equation}\label{eq:Leavitt-frames}
\begin{aligned}
 U_p&=(X_{1\bullet})^{\star p},&
 V_p&=(Y_{\bullet1})^{\star p},\\
 U_{-p}&=(Y_{1\bullet})^{\star p},&
 V_{-p}&=(X_{\bullet1})^{\star p}.
\end{aligned}
\end{equation}
 The bases of the positive and negative $\star$-products are $n$ and $m$, respectively. For $p=0$ set $U_0=V_0=(1)$. Clearly
 $U_p\in\Mat_{1\times d_p}(L_p)$ and $V_p\in\Mat_{d_p\times 1}(L_{-p})$ for all $p\in \Z$, where 
\[
 d_p=\begin{cases}n^p&p\geq0,\\m^{-p}&p<0.\end{cases}
\]
It follows from \eqref{stronglygraded} that $U_pV_p=1$ for all $p\in\Z$.

\subsection{\texorpdfstring{The zero component of $L_K(m,n)$}{The zero component of L(m,n) over K}}
Let $R$ be the algebra constructed in Section~\ref{sec:general}
for the presentation \eqref{eq:Leavittstandard} and the frames
\eqref{eq:Leavitt-frames} (we choose the canonical representatives of the frame entries in $K\langle x_{ij},y_{ji}\rangle$). Theorem \ref{thm:Leavittfactor} below follows from Theorem \ref{thm:homogeneous-presentation}.

\begin{theorem}\label{thm:Leavittfactor}
$L_K(m,n)_0\cong R$.
\end{theorem}

For $p\geq1$ and $1\leq k,\ell\leq m$, let $E_p^{k\ell}$ be an
$n^p\times n^p$ matrix of symbols. Let $S(m,n)$ be the algebra generated by the entries of these matrices subject to the entrywise relations
\begin{equation}\label{eq:Amn}
 E_p^{k\ell}E_p^{ab}=\delta_{\ell a}E_p^{kb},\qquad
 \sum_{k=1}^mE_p^{kk}=\diag_n(E_{p-1}^{11}).
\end{equation}
Here we set $E_0^{11}=(1)$. Define the algebra $S(n,m)$ analogously and set 
\[
 S=S(m,n)*_K S(n,m).
\]

\begin{theorem}\label{thm:Leavitt-coproduct}
There is an algebra isomorphism $R\cong S$. Consequently,
\[
 L_K(m,n)_0\cong S(m,n)*_K S(n,m).
\]
\end{theorem}
\begin{proof} 

\emph{The homomorphism $\mu:S\to R$.}
We leave it to the reader to check that there is an algebra homomorphism $\theta:S\to L_0$ such that entrywise
\begin{equation}\label{eq:Leavitt-theta}
\begin{aligned}
 \theta(E_p^{k\ell})
 &=\diag_n(V_{p-1})Y_{\bullet k}X_{\ell\bullet}\diag_n(U_{p-1}),\\
 \theta(F_p^{k\ell})
 &=\diag_m(V_{-p+1})X_{\bullet k}Y_{\ell\bullet}\diag_m(U_{-p+1})
 \qquad(p\geq1)
\end{aligned}
\end{equation}
where $E_p^{k\ell}$ are the generating matrices of $S(m,n)$ and $F_p^{k\ell}$ the generating matrices of $S(n,m)$. Let $\beta:L_0\to R$ be the isomorphism constructed in the proof of
Theorem~\ref{thm:homogeneous-presentation}. We set
\[
 \mu=\beta\circ\theta:S\longrightarrow R.
\]
 
\emph{The homomorphism $\nu:R\to S$.}
For $p\geq1$ and $1\leq j\leq n$ let
\[
 B_p^j=(0_{n^{p-1}},\dots,0_{n^{p-1}},I_{n^{p-1}},
        0_{n^{p-1}},\dots,0_{n^{p-1}})^{\mathsf T}
 \in\Mat_{n^p\times n^{p-1}}(K),
\]
with the identity in the $j$-th position. Similarly, for $1\leq i\leq m$ let
\[
 C_p^i=(0_{m^{p-1}},\dots,0_{m^{p-1}},I_{m^{p-1}},
        0_{m^{p-1}},\dots,0_{m^{p-1}})^{\mathsf T}
 \in\Mat_{m^p\times m^{p-1}}(K),
\]
with the identity in the $i$-th position. Define the images of the generators of $R$ by
\begin{equation}\label{eq:Leavitttransitions}
\begin{aligned}
 \nu(Q_p)&=E_p^{11},&
 \nu(Q_{-p})&=F_p^{11},\\
 \nu(T^{x_{ij}}_{p-1})&=E_p^{1i}B_p^j,&
 \nu(T^{x_{ij}}_{-p})&=(C_p^i)^{\mathsf T}F_p^{j1},\\
 \nu(T^{y_{ji}}_{p})&=(B_p^j)^{\mathsf T}E_p^{i1},\quad &
 \nu(T^{y_{ji}}_{-p+1})&=F_p^{1j}C_p^i
 \qquad(p\geq1).
\end{aligned}
\end{equation}
We have to check that the defining relations \eqref{eq:idempotent}-\eqref{eq:general-normalization2} of $R$ are preserved.

Set $D_p=\sum_kE_p^{kk}=\diag_n(E_{p-1}^{11})$.
Relations \eqref{eq:Amn} imply
$D_pE_p^{k\ell}=E_p^{k\ell}=E_p^{k\ell}D_p$.
Together with $E_p^{11}E_p^{1i}=E_p^{1i}$ and
$E_p^{i1}E_p^{11}=E_p^{i1}$, these identities verify
\eqref{eq:idempotent} and \eqref{eq:general-support} at positive
levels. At negative levels, the corresponding calculation with $F$ and $C$
verifies \eqref{eq:idempotent} and \eqref{eq:general-support}. 

Relation \eqref{eq:general-lifted-relations}, applied to
\eqref{eq:Leavittstandard}, states that
\begin{equation}\label{eq:Leavitt-lifted}
 \sum_jT^{x_{ij}}_tT^{y_{jk}}_{t+1}
   =\delta_{ik}Q_{t+1},\qquad
 \sum_iT^{y_{ji}}_{t+1}T^{x_{i\ell}}_t
   =\delta_{j\ell}Q_t
 \qquad(t\in\Z).
\end{equation}
For $t=p-1\geq 0$ the two equations in \eqref{eq:Leavitt-lifted} hold because
\[
\begin{aligned}
 \sum_j(E_p^{1i}B_p^j)((B_p^j)^{\mathsf T}E_p^{k1})
 &=E_p^{1i}E_p^{k1}=\delta_{ik}E_p^{11},\\
 \sum_i((B_p^j)^{\mathsf T}E_p^{i1})(E_p^{1i}B_p^\ell)
 &=(B_p^j)^{\mathsf T}D_pB_p^\ell
   =\delta_{j\ell}E_{p-1}^{11}.
\end{aligned}
\]
For $t<0$, the analogous calculation with $F$ and $C$
verifies both equations in \eqref{eq:Leavitt-lifted}.

It remains to verify \eqref{eq:general-normalization1} and
\eqref{eq:general-normalization2}. For each $t\in\Z$ define
\[
 H_t=
 \begin{pmatrix}
  \nu(T^{\tilde u_{t,1}}_0)&\cdots&
  \nu(T^{\tilde u_{t,d_t}}_0)
 \end{pmatrix},
 \qquad
 K_t=
 \begin{pmatrix}
  \nu(T^{\tilde v_{t,1}}_t)\\
  \vdots\\
  \nu(T^{\tilde v_{t,d_t}}_t)
 \end{pmatrix}.
\]
We have to show that
\[
 H_t=K_t=\nu(Q_t)
\]
for all $t\in \Z$. We prove these equalities simultaneously at $t=p$ and $t=-p$
by induction on $p\geq0$. At $p=0$ all three matrices are $(1)$.

Suppose $p\geq1$. By \eqref{eq:Leavitt-frames}, the words in
the $j$-th block of $U_p$ are
\[
 x_{1j}\tilde u_{p-1,1},\ldots,
 x_{1j}\tilde u_{p-1,n^{p-1}}.
\]
Clearly
\[
 \nu(T^{x_{1j}\tilde u_{p-1,k}}_0)=\nu(T^{x_{1j}}_{p-1}T^{\tilde u_{p-1,k}}_0)
 =\nu(T^{x_{1j}}_{p-1})\nu(T^{\tilde u_{p-1,k}}_0)
\]
for all $1\leq k\leq n^{p-1}$. 
The second factor on the right hand side is the $k$-th column of $H_{p-1}$.
Thus the $j$-th block column of $H_p$ is
$\nu(T^{x_{1j}}_{p-1})H_{p-1}$. Similarly, the $j$-th block row of $K_p$ is
$K_{p-1}\nu(T^{y_{j1}}_p)$.
Using \eqref{eq:Leavitttransitions}, we therefore obtain
\[
\begin{aligned}
 H_pB_p^j
 &=\nu(T^{x_{1j}}_{p-1})H_{p-1}
  =E_p^{11}B_p^jH_{p-1},\\
 (B_p^j)^{\mathsf T}K_p
 &=K_{p-1}\nu(T^{y_{j1}}_p)
  =K_{p-1}(B_p^j)^{\mathsf T}E_p^{11}.
\end{aligned}
\]
Since these identities hold for every $j$, it follows that
\[
 H_p=E_p^{11}\diag_n(H_{p-1}),
 \qquad
 K_p=\diag_n(K_{p-1})E_p^{11}.
\]
The induction hypothesis states that
$H_{p-1}=K_{p-1}=E_{p-1}^{11}$. Therefore,
\[
 H_p=E_p^{11}D_p=E_p^{11},
 \qquad
 K_p=D_pE_p^{11}=E_p^{11}.
\]

At level $-p$, the corresponding words are
$y_{1i}\tilde u_{-p+1,k}$ and
$\tilde v_{-p+1,k}x_{i1}$.
Repeating the calculation with $C_p^i$, we obtain
\[
 H_{-p}=F_p^{11}\diag_m(H_{-p+1}),
 \qquad
 K_{-p}=\diag_m(K_{-p+1})F_p^{11}.
\]
The induction hypothesis and the defining relations for the
$F$-matrices imply
\[
 H_{-p}=K_{-p}=F_p^{11}.
\]

We have shown that all defining relations of $R$ are preserved. Thus
$\nu:R\to S$ is a well-defined algebra homomorphism.

\emph{The composites.}
Let $\alpha=\beta^{-1}:R\to L_0$ be the isomorphism constructed in
the proof of Theorem~\ref{thm:homogeneous-presentation}.
We first show that $\theta\nu=\alpha$.
The frame formulas \eqref{eq:Leavitt-frames} and the definition
\eqref{eq:Leavitt-theta} imply for $p\geq1$,
\[
\begin{aligned}
 \theta\nu(Q_p)&=\theta(E_p^{11})=V_pU_p=\alpha(Q_p),\\
 \theta\nu(Q_{-p})&=\theta(F_p^{11})=V_{-p}U_{-p}=\alpha(Q_{-p}).
\end{aligned}
\]
For the transition matrices, \eqref{eq:Leavitt-theta},
\eqref{eq:Leavitttransitions}, and \eqref{eq:Leavitt-frames} imply for $p\geq 1$
\[
\begin{aligned}
 \theta\nu(T^{x_{ij}}_{p-1})
 &=\theta(E_p^{1i})B_p^j=V_px_{ij}U_{p-1}=\alpha(T^{x_{ij}}_{p-1}),\\
 \theta\nu(T^{y_{ji}}_p)
 &=(B_p^j)^{\mathsf T}\theta(E_p^{i1})=V_{p-1}y_{ji}U_p=\alpha(T^{y_{ji}}_p),\\
 \theta\nu(T^{x_{ij}}_{-p})
 &=(C_p^i)^{\mathsf T}\theta(F_p^{j1})=V_{-p+1}x_{ij}U_{-p}= \alpha(T^{x_{ij}}_{-p}),\\
 \theta\nu(T^{y_{ji}}_{-p+1})
 &=\theta(F_p^{1j})C_p^i=V_{-p}y_{ji}U_{-p+1}=\alpha(T^{y_{ji}}_{-p+1}).
\end{aligned}
\]
It follows that $\theta\nu=\alpha$ and hence $\mu\nu=\beta\theta\nu=\beta\alpha=\id_R$.

To obtain the other composite, observe that \eqref{eq:Leavitttransitions}
and \eqref{eq:Amn} imply
\[
\begin{aligned}
 \nu(T^{y_{jk}}_pT^{x_{\ell r}}_{p-1})
 &=(B_p^j)^{\mathsf T}E_p^{k1}E_p^{1\ell}B_p^r
  =(B_p^j)^{\mathsf T}E_p^{k\ell}B_p^r,\\
 \nu(T^{x_{ik}}_{-p}T^{y_{\ell s}}_{-p+1})
 &=(C_p^i)^{\mathsf T}F_p^{k1}F_p^{1\ell}C_p^s
  =(C_p^i)^{\mathsf T}F_p^{k\ell}C_p^s.
\end{aligned}
\]
As $j,r$ range from $1$ to $n$ and $i,s$ range from $1$ to $m$,
the right-hand sides comprise all blocks of the generating
matrices $E_p^{k\ell}$ and $F_p^{k\ell}$. Every generator of $S$
therefore lies in $\nu(R)$ and hence $\nu$ is surjective. For every
$s=\nu(r)\in S$, we have $\nu\mu(s)=\nu\mu\nu(r)=\nu(r)=s$.
Hence $\nu\mu=\id_S$ and thus $\mu$ and $\nu$ are inverse to each other.
\end{proof}

\section{The zero component of a hypergraph Leavitt path algebra}\label{sec:hyper}
\subsection{\texorpdfstring{The Leavitt path algebra $L_K(H)$}{The hypergraph Leavitt path algebra}}\label{subsec:hyper-definition}
A \emph{hypergraph} $H$ consists of a vertex set $H^0$, a
hyperedge set $H^1$ and, for each $h\in H^1$, finite nonempty
indexed families
\[
 (s(h)_i)_{i\in I_h},\qquad (r(h)_j)_{j\in J_h}
\]
of source and range vertices; see \cite{Hyper}. A vertex can
appear more than once in the source or range family of a hyperedge. We call a vertex $v\in H^0$ a \textit{sink} if $v\neq s(h)_i$ for all $h\in H^1$ and $i\in I_h$, and a \textit{source} if $v\neq r(h)_j$ for all $h\in H^1$ and $j\in J_h$.

Let $H$ be a finite hypergraph without sinks and sources.
The \emph{Leavitt path algebra} $L=L_K(H)$ is generated by
\[
 \{v,h_{ij},h^*_{ij}\mid v\in H^0,\ h\in H^1,\ i\in I_h,\ j\in J_h\}.
\]
The vertices are orthogonal idempotents with sum $1$, and
\[
 h_{ij}=s(h)_ih_{ij}r(h)_j,\qquad
 h^*_{ij}=r(h)_jh^*_{ij}s(h)_i
 \qquad(h\in H^1,\ i\in I_h,\ j\in J_h).
\]
The remaining defining relations are
\begin{equation}\label{eq:hyper-relations}
\begin{aligned}
 \sum_{j\in J_h}h_{ij}h^*_{i'j}&=\delta_{ii'}s(h)_i
       &&(h\in H^1;\ i,i'\in I_h),\\
 \sum_{i\in I_h}h^*_{ij}h_{ij'}&=\delta_{jj'}r(h)_j
       &&(h\in H^1;\ j,j'\in J_h).
\end{aligned}
\end{equation}
We give $L_K(H)$ its \emph{canonical $\Z$-grading} defined by
$\delta(v)=0$, $\delta(h_{ij})=1$ and $\delta(h^*_{ij})=-1$
for all $v\in H^0$, $h\in H^1$, $i\in I_h$ and $j\in J_h$.

We use disjoint index
sets for different hyperedges and put
\[
 I=\bigsqcup_{h\in H^1}I_h,\qquad
 J=\bigsqcup_{h\in H^1}J_h.
\]
Moreover, we fix orders on $I$ and $J$ and let $X$ and $Y$ be the $I\times J$ and $J\times I$ matrices
with entries
\[
 X_{ij}=\begin{cases}h_{ij},&i\in I_h,\ j\in J_h\text{ for some }h,\\
 0,&\text{otherwise},\end{cases}\qquad
 Y_{ji}=\begin{cases}h^*_{ij},&i\in I_h,\ j\in J_h\text{ for some }h,\\
 0,&\text{otherwise}.\end{cases}
\]
Then \eqref{eq:hyper-relations} is equivalent to $XY$ and $YX$
being diagonal, with diagonal entries
\[
 (XY)_{ii}=s(h)_i\quad(i\in I_h),\qquad
 (YX)_{jj}=r(h)_j\quad(j\in J_h)
 \qquad(h\in H^1).
\]
For each vertex $v$, choose hyperedges $h,g\in H^1$ and indices
$\alpha_v\in I_h$, $\beta_v\in J_g$ such that $s(h)_{\alpha_v}=v=r(g)_{\beta_v}$.
Since
\begin{equation}\label{eq:hyper-strong-grading}
\begin{aligned}
 1&=\sum_{v\in H^0}X_{\alpha_v\bullet}Y_{\bullet \alpha_v}
       \in L_1L_{-1},\\
 1&=\sum_{v\in H^0}Y_{\beta_v\bullet}X_{\bullet \beta_v}
       \in L_{-1}L_1,
\end{aligned}
\end{equation}
the algebra $L_K(H)$ is strongly graded.

\subsection{Choosing the frames}\label{subsec:hyper-frames}
For $p\geq1$, choose the frames
\begin{equation}\label{eq:hyper-frames}
\begin{aligned}
 U_p&=\left(\sum_{v\in H^0}X_{\alpha_v\bullet}\right)^{\star p},
 &V_p&=\left(\sum_{v\in H^0}Y_{\bullet \alpha_v}\right)^{\star p},\\
 U_{-p}&=\left(\sum_{v\in H^0}Y_{\beta_v\bullet}\right)^{\star p},
 &V_{-p}&=\left(\sum_{v\in H^0}X_{\bullet \beta_v}\right)^{\star p}.
\end{aligned}
\end{equation}
The bases of the positive and negative $\star$-products are
$b=|J|$ and $a=|I|$, respectively. For $p=0$, set $U_0=V_0=(1)$.
Clearly
\[
 U_p\in\Mat_{1\times d_p}(L_p),\qquad
 V_p\in\Mat_{d_p\times1}(L_{-p})
 \qquad(p\in\Z),
\]
where
\[
 d_p=\begin{cases}b^p&p\geq0,\\a^{-p}&p<0.\end{cases}
\]
The matrix identities for $XY$ and $YX$ imply
$U_1V_1=U_{-1}V_{-1}=1$, and hence $U_pV_p=1$
for every $p\in\Z$.

\subsection{\texorpdfstring{The zero component of $L_K(H)$}{The zero component of the hypergraph Leavitt path algebra}}\label{subsec:hyper-zero}
Consider $A=K^{H^0}$ as an algebra with coordinatewise
addition and multiplication and diagonal structure map $K\to A$.
We denote its standard coordinate idempotents by the corresponding
vertices. Thus
\[
 A=\bigoplus_{v\in H^0}Kv,\qquad
 vu=\delta_{v,u}v,\qquad 1_A=\sum_{v\in H^0}v.
\]
Regard $L_0$ as an $A$-algebra via the algebra homomorphism
$A\to L_0$ mapping $v\mapsto v$.

Let $R$ be the algebra constructed in Section~\ref{sec:general}
for the homogeneous presentation in \S\ref{subsec:hyper-definition}
and the frames in \S\ref{subsec:hyper-frames} (we choose the canonical representatives of the frame entries). Regard $R$ as an $A$-algebra via
\[
 A\longrightarrow R,\qquad v\longmapsto T^v_0.
\]
The lifted relation $\sum_vv=1$ states that
$\sum_vT^v_p=Q_p$ for every $p\in\Z$.
Theorem~\ref{thm:hypercoproduct} below follows from the proof of
Theorem~\ref{thm:homogeneous-presentation}, since the isomorphism
$\alpha:R\to L_0$ constructed there sends $T^v_0$ to $v$.

\begin{theorem}\label{thm:hypercoproduct}
There is an $A$-algebra isomorphism
\[
 L_K(H)_0\cong R.
\]
\end{theorem}

For $p\geq1$, $h,g\in H^1$, $i\in I_h$ and $k\in I_g$ with $s(h)_i=s(g)_k$, let $E_p^{ik}$
be a $b^p\times b^p$ matrix of symbols. Regard each matrix as a
$b\times b$ block matrix whose blocks have size
$b^{p-1}\times b^{p-1}$. Index its block rows and block columns
by the elements of $J$, using the fixed order.
Let $H^+_{h,p}$ be the $b^p\times b^p$ block diagonal matrix, with the same block sizes and indexing, whose $j$-th
diagonal block is $E_{p-1}^{\alpha_{r(h)_j}\alpha_{r(h)_j}}$ if
$j\in J_h$, and zero otherwise
(here $E_0^{\alpha_v\alpha_v}=(v)$).
Let $S^+(H)$ be the $A$-algebra generated by the entries
of the matrices $E_p^{ik}$ subject to the entrywise relations
\begin{equation}\label{eq:hyper-positive}
 E_p^{ik}E_p^{\ell t}=\delta_{k\ell}E_p^{it},\qquad
 \sum_{i\in I_h}E_p^{ii}=H^+_{h,p}.
\end{equation}
In the first relation the product is zero when $k\ne\ell$;
when $k=\ell$, its right-hand side is defined.

For $p\geq1$, $h,g\in H^1$, $j\in J_h$ and $k\in J_g$ with
$r(h)_j=r(g)_k$, let $F_p^{jk}$ be an $a^p\times a^p$ matrix
of symbols. Regard each matrix as an $a\times a$ block matrix
whose blocks have size $a^{p-1}\times a^{p-1}$. Index its block
rows and block columns by the elements of $I$, using the fixed
order.
Let $H^-_{h,p}$ be the $a^p\times a^p$ block diagonal matrix,
with the same block sizes and indexing, whose $i$-th diagonal
block is $F_{p-1}^{\beta_{s(h)_i}\beta_{s(h)_i}}$ if $i\in I_h$,
and zero otherwise (here $F_0^{\beta_v\beta_v}=(v)$).
Let $S^-(H)$ be the $A$-algebra generated by the entries
of the matrices $F_p^{jk}$ subject to the entrywise relations
\begin{equation}\label{eq:hyper-negative}
 F_p^{jk}F_p^{\ell t}=\delta_{k\ell}F_p^{jt},\qquad
 \sum_{j\in J_h}F_p^{jj}=H^-_{h,p}.
\end{equation}
In the first relation the product is zero when $k\ne\ell$;
when $k=\ell$, its right-hand side is defined.
Set
\[
 S=S^+(H)*_AS^-(H).
\]

\begin{theorem}\label{thm:hyper-recursive}
There is an $A$-algebra isomorphism $R\cong S$. Consequently,
\[
 L_K(H)_0\cong S^+(H)*_A S^-(H).
\]
\end{theorem}
\begin{proof}
\emph{The homomorphism $\mu:S\to R$.}
We leave it to the reader to check that there is an $A$-algebra
homomorphism $\theta:S\to L_0$ such that entrywise
\begin{equation}\label{eq:hyper-theta}
\begin{aligned}
 \theta(E_p^{ik})
 &=\diag_b(V_{p-1})Y_{\bullet i}X_{k\bullet}\diag_b(U_{p-1}),\\
 \theta(F_p^{jk})
 &=\diag_a(V_{-p+1})X_{\bullet j}Y_{k\bullet}\diag_a(U_{-p+1})
 \qquad(p\geq1).
\end{aligned}
\end{equation}
Let $\beta:L_0\to R$ be the isomorphism constructed in the proof
of Theorem~\ref{thm:homogeneous-presentation}. We set
\[
 \mu=\beta\circ\theta:S\longrightarrow R.
\]

\emph{The homomorphism $\nu:R\to S$.}
For $p\geq1$ and $j\in J$, let
\[
 B_p^j=(0_{b^{p-1}},\ldots,I_{b^{p-1}},\ldots,0_{b^{p-1}})^{\mathsf T}
 \in\Mat_{b^p\times b^{p-1}}(K),
\]
with the identity in position $j$. Similarly, for $i\in I$, let
\[
 C_p^i=(0_{a^{p-1}},\ldots,I_{a^{p-1}},\ldots,0_{a^{p-1}})^{\mathsf T}
 \in\Mat_{a^p\times a^{p-1}}(K),
\]
with the identity in position $i$. Set
\[
 q_0=(1),\qquad q_p=\sum_vE_p^{\alpha_v\alpha_v},\qquad
 q_{-p}=\sum_vF_p^{\beta_v\beta_v}\quad(p\geq1).
\]
Define $\nu(T^v_0)=v$ for every $v\in H^0$, and the images of the other generators by
 the following equations, where $p\geq 1$, $h\in H^1$, $i\in I_h$ and $j\in J_h$:
\begin{equation}\label{eq:hyper-transitions}
\begin{aligned}
 \nu(Q_p)&=q_p,&\nu(Q_{-p})&=q_{-p},\\
 \nu(T^v_p)&=E_p^{\alpha_v\alpha_v},&
 \nu(T^v_{-p})&=F_p^{\beta_v\beta_v},\\
 \nu(T^{h_{ij}}_{p-1})&=E_p^{\alpha_{s(h)_i},i}B_p^j,&
 \nu(T^{h_{ij}}_{-p})&=(C_p^i)^{\mathsf T}F_p^{j,\beta_{r(h)_j}},\\
 \nu(T^{h^*_{ij}}_p)&=(B_p^j)^{\mathsf T}E_p^{i,\alpha_{s(h)_i}},&
 \nu(T^{h^*_{ij}}_{-p+1})&=F_p^{\beta_{r(h)_j},j}C_p^i.
\end{aligned}
\end{equation}
We have to check that the defining relations \eqref{eq:idempotent}--\eqref{eq:general-normalization2}
of $R$ are preserved.

\emph{Relation \eqref{eq:idempotent}.}
By \eqref{eq:hyper-positive}, the matrices
$E_p^{\alpha_v\alpha_v}$, as $v$ varies over $H^0$, are
pairwise orthogonal idempotents. The same holds for the
matrices $F_p^{\beta_v\beta_v}$ by
\eqref{eq:hyper-negative}. Hence
\[
 q_p^2=q_p,\qquad q_{-p}^2=q_{-p}\qquad(p\geq1).
\]
Since $\nu(Q_t)=q_t$, relation \eqref{eq:idempotent}
is preserved.

\emph{Relation \eqref{eq:general-support}.}
First, \eqref{eq:hyper-positive} and
\eqref{eq:hyper-negative} imply
\begin{equation}\label{eq:hyper-recursive-support}
\begin{aligned}
 H^+_{h,p}E_p^{ik}&=E_p^{ik}=E_p^{ik}H^+_{g,p},\\
 H^-_{h,p}F_p^{j\ell}&=F_p^{j\ell}=F_p^{j\ell}H^-_{g,p}.
\end{aligned}
\end{equation}
Here the first line holds for $i\in I_h$, $k\in I_g$ with
$s(h)_i=s(g)_k$, and the second for $j\in J_h$, $\ell\in J_g$
with $r(h)_j=r(g)_\ell$. For $p\geq1$, set
\[
 D_p=\diag_b(q_{p-1}),\qquad
 D_{-p}=\diag_a(q_{-p+1}).
\]
Each nonzero diagonal block of $H^+_{h,p}$ is one of the
orthogonal idempotents whose sum is $q_{p-1}$.
Similarly, each nonzero diagonal block of $H^-_{h,p}$ is
one of the orthogonal idempotents whose sum is $q_{-p+1}$.
For $p=1$, we have $D_1=I_b$ and $D_{-1}=I_a$,
so multiplication by these matrices leaves
$H^+_{h,1}$ and $H^-_{h,1}$ unchanged. Thus
\[
\begin{aligned}
 D_pH^+_{h,p}&=H^+_{h,p}=H^+_{h,p}D_p,\\
 D_{-p}H^-_{h,p}&=H^-_{h,p}=H^-_{h,p}D_{-p}.
\end{aligned}
\]
Together with \eqref{eq:hyper-recursive-support}, these
identities imply
\[
 D_pE_p^{ik}=E_p^{ik}=E_p^{ik}D_p,\qquad
 D_{-p}F_p^{j\ell}=F_p^{j\ell}=F_p^{j\ell}D_{-p}
\]
for all the generating matrices. In particular,
\begin{equation}\label{eq:hyper-q-diagonal}
 q_pD_p=q_p=D_pq_p,\qquad
 q_{-p}D_{-p}=q_{-p}=D_{-p}q_{-p}.
\end{equation}

For a vertex generator $u$ and $p\geq1$,
\eqref{eq:hyper-transitions} and
\eqref{eq:hyper-positive} imply
\[
\begin{aligned}
 q_p\nu(T^u_p)q_p
 &=\left(\sum_vE_p^{\alpha_v\alpha_v}\right)
   E_p^{\alpha_u\alpha_u}
   \left(\sum_wE_p^{\alpha_w\alpha_w}\right)\\
 &=E_p^{\alpha_u\alpha_u}
  =\nu(T^u_p).
\end{aligned}
\]
Similarly, \eqref{eq:hyper-negative} implies
\[
 q_{-p}\nu(T^u_{-p})q_{-p}
 =F_p^{\beta_u\beta_u}
 =\nu(T^u_{-p}).
\]
At level zero, $q_0=(1)$ and $\nu(T^u_0)=u$, so
\[
 q_0\nu(T^u_0)q_0=1_Au1_A=u=\nu(T^u_0).
\]
For the other generators, fix $h\in H^1$, $i\in I_h$ and
$j\in J_h$, and put $v=s(h)_i$ and $w=r(h)_j$.
Using \eqref{eq:hyper-transitions}, block multiplication,
and the identities just established, we obtain
\[
\begin{aligned}
 q_pE_p^{\alpha_v i}B_p^jq_{p-1}
 &
  =E_p^{\alpha_v i}B_p^j,\\
 q_{p-1}(B_p^j)^{\mathsf T}E_p^{i\alpha_v}q_p
 &
  =(B_p^j)^{\mathsf T}E_p^{i\alpha_v},\\
 q_{-p+1}(C_p^i)^{\mathsf T}F_p^{j\beta_w}q_{-p}
 &
  =(C_p^i)^{\mathsf T}F_p^{j\beta_w},\\
 q_{-p}F_p^{\beta_wj}C_p^iq_{-p+1}
 &
  =F_p^{\beta_wj}C_p^i.
\end{aligned}
\]
These are precisely the four cases of
\[
 q_{t+\delta(x)}\nu(T^x_t)q_t=\nu(T^x_t).
\]
Thus \eqref{eq:general-support} is preserved.

\emph{Relation \eqref{eq:general-lifted-relations}.}
We check the defining relations of $L$ at every level.

For the vertex relations, \eqref{eq:hyper-positive},
\eqref{eq:hyper-negative}, and the vertex relations in $A$
imply
\[
 \nu(T^v_t)\nu(T^w_t)=\delta_{vw}\nu(T^v_t),
 \qquad
 \sum_{v\in H^0}\nu(T^v_t)=q_t=\nu(Q_t)
 \qquad(t\in\Z).
\]
Thus $\nu$ preserves the instances of
\eqref{eq:general-lifted-relations} corresponding to
$vw=\delta_{vw}v$ and $\sum_vv=1$.

Next, fix $h\in H^1$, $i\in I_h$ and $j\in J_h$, and again
put $v=s(h)_i$ and $w=r(h)_j$. The relations
$h_{ij}=vh_{ij}w$ and $h^*_{ij}=wh^*_{ij}v$ require
\[
\begin{aligned}
 \nu(T^{h_{ij}}_t)
 &=\nu(T^v_{t+1})\nu(T^{h_{ij}}_t)\nu(T^w_t),\\
 \nu(T^{h^*_{ij}}_t)
 &=\nu(T^w_{t-1})\nu(T^{h^*_{ij}}_t)\nu(T^v_t)
 \qquad(t\in\Z).
\end{aligned}
\]
The definitions of $H^+_{h,p}$ and $H^-_{h,p}$ imply
\begin{equation}\label{eq:hyper-block-identities}
\begin{aligned}
 H^+_{h,p}B_p^j
 &=B_p^jE_{p-1}^{\alpha_w\alpha_w},&
 (B_p^j)^{\mathsf T}H^+_{h,p}
 &=E_{p-1}^{\alpha_w\alpha_w}(B_p^j)^{\mathsf T},\\
 H^-_{h,p}C_p^i
 &=C_p^iF_{p-1}^{\beta_v\beta_v},&
 (C_p^i)^{\mathsf T}H^-_{h,p}
 &=F_{p-1}^{\beta_v\beta_v}(C_p^i)^{\mathsf T}.
\end{aligned}
\end{equation}
These identities also hold for $p=1$, under the conventions
$E_0^{\alpha_u\alpha_u}=F_0^{\beta_u\beta_u}=(u)$.
Using \eqref{eq:hyper-positive},
\eqref{eq:hyper-negative},
\eqref{eq:hyper-recursive-support} and
\eqref{eq:hyper-block-identities}, we obtain
\[
\begin{aligned}
 E_p^{\alpha_v\alpha_v}E_p^{\alpha_v i}
 B_p^jE_{p-1}^{\alpha_w\alpha_w}
 &=E_p^{\alpha_v i}H^+_{h,p}B_p^j
  =E_p^{\alpha_v i}B_p^j,\\
 E_{p-1}^{\alpha_w\alpha_w}(B_p^j)^{\mathsf T}
 E_p^{i\alpha_v}E_p^{\alpha_v\alpha_v}
 &=(B_p^j)^{\mathsf T}H^+_{h,p}E_p^{i\alpha_v}
  =(B_p^j)^{\mathsf T}E_p^{i\alpha_v},\\
 F_{p-1}^{\beta_v\beta_v}(C_p^i)^{\mathsf T}
 F_p^{j\beta_w}F_p^{\beta_w\beta_w}
 &=(C_p^i)^{\mathsf T}H^-_{h,p}F_p^{j\beta_w}
  =(C_p^i)^{\mathsf T}F_p^{j\beta_w},\\
 F_p^{\beta_w\beta_w}F_p^{\beta_wj}
 C_p^iF_{p-1}^{\beta_v\beta_v}
 &=F_p^{\beta_wj}H^-_{h,p}C_p^i
  =F_p^{\beta_wj}C_p^i.
\end{aligned}
\]
By \eqref{eq:hyper-transitions}, these are the four required
cases.

It remains to check the instances of
\eqref{eq:general-lifted-relations} corresponding to
\eqref{eq:hyper-relations}, namely
\begin{equation}\label{eq:hyper-universal-relations}
\begin{aligned}
 \sum_{j\in J_h}T^{h_{ij}}_tT^{h^*_{kj}}_{t+1}
 &=\delta_{ik}T^{s(h)_i}_{t+1},\\
 \sum_{i\in I_h}T^{h^*_{ij}}_{t+1}T^{h_{i\ell}}_t
 &=\delta_{j\ell}T^{r(h)_j}_t
 \qquad(t\in\Z).
\end{aligned}
\end{equation}
Here $i,k\in I_h$ in the first line and $j,\ell\in J_h$
in the second. For $t=p-1\geq0$, preservation of these
relations follows from
\[
\begin{aligned}
 E_p^{\alpha_{s(h)_i},i}
 \Bigl(\sum_{j\in J_h}B_p^j(B_p^j)^{\mathsf T}\Bigr)
 E_p^{k,\alpha_{s(h)_k}}
 &=E_p^{\alpha_{s(h)_i},i}E_p^{k,\alpha_{s(h)_k}}
  =\delta_{ik}E_p^{\alpha_{s(h)_i}\alpha_{s(h)_i}},\\
 (B_p^j)^{\mathsf T}
 \Bigl(\sum_{i\in I_h}E_p^{ii}\Bigr)B_p^\ell
 &=(B_p^j)^{\mathsf T}H^+_{h,p}B_p^\ell
  =\delta_{j\ell}E_{p-1}^{\alpha_{r(h)_j}\alpha_{r(h)_j}}.
\end{aligned}
\]

For $t=-p<0$, the corresponding calculations are
\[
\begin{aligned}
 (C_p^i)^{\mathsf T}
 \Bigl(\sum_{j\in J_h}F_p^{jj}\Bigr)C_p^k
 &=(C_p^i)^{\mathsf T}H^-_{h,p}C_p^k
  =\delta_{ik}F_{p-1}^{\beta_{s(h)_i}\beta_{s(h)_i}},\\
 F_p^{\beta_{r(h)_j},j}
 \Bigl(\sum_{i\in I_h}C_p^i(C_p^i)^{\mathsf T}\Bigr)
 F_p^{\ell,\beta_{r(h)_\ell}}
 &=F_p^{\beta_{r(h)_j},j}F_p^{\ell,\beta_{r(h)_\ell}}
  =\delta_{j\ell}F_p^{\beta_{r(h)_j}\beta_{r(h)_j}}.
\end{aligned}
\]
This completes the verification of
\eqref{eq:general-lifted-relations}.

\emph{Relation \eqref{eq:general-normalization1}.}
For each $t\in\Z$ define the matrix
\[
 H_t=
 \begin{pmatrix}
  \nu(T^{\tilde u_{t,1}}_{0,t})&\cdots&
  \nu(T^{\tilde u_{t,d_t}}_{0,t})
 \end{pmatrix}.
\]
Since the $k$-th column of $q_t$ is $q_t\varepsilon_k$,
relation \eqref{eq:general-normalization1} is equivalent to
$H_t=q_t$.

We prove $H_p=q_p$ and $H_{-p}=q_{-p}$ simultaneously by
induction on $p\geq0$. At $p=0$, both matrices are $(1)$.
Suppose $p\geq1$ and $j\in J_h$. By
\eqref{eq:hyper-frames}, the polynomial expressions in the
$j$-th block of $U_p$ are
\[
 \left(\sum_{v:\,\alpha_v\in I_h}h_{\alpha_vj}\right)
 \tilde u_{p-1,k}
 \qquad(1\leq k\leq b^{p-1}).
\]
Clearly
\[
 \nu(T^{h_{\alpha_vj}\tilde u_{p-1,k}}_{0,p})
 =
 \nu(T^{h_{\alpha_vj}}_{p-1})
 \nu(T^{\tilde u_{p-1,k}}_{0,p-1}).
\]
The second factor is the $k$-th column of $H_{p-1}$.
Consequently,
\[
\begin{aligned}
 H_pB_p^j
 &=\sum_{v:\,\alpha_v\in I_h}
       \nu(T^{h_{\alpha_vj}}_{p-1})H_{p-1}\\
 &=\sum_{v:\,\alpha_v\in I_h}
       E_p^{\alpha_v\alpha_v}B_p^jH_{p-1}
  =q_pB_p^jH_{p-1}.
\end{aligned}
\]
The second equality uses \eqref{eq:hyper-transitions}.
For the last equality, if $\alpha_v\in I_g$ with $g\ne h$,
then \eqref{eq:hyper-recursive-support} and the definition of
$H^+_{g,p}$ imply
\[
 E_p^{\alpha_v\alpha_v}B_p^j
 =E_p^{\alpha_v\alpha_v}H^+_{g,p}B_p^j=0.
\]
Since the block-column identity holds for every $j\in J$,
\[
 H_p=q_p\diag_b(H_{p-1}).
\]

At level $-p$, for $i\in I_h$, the corresponding frame entries
are
\[
 \left(\sum_{v:\,\beta_v\in J_h}h^*_{i\beta_v}\right)
 \tilde u_{-p+1,k}.
\]
As above
\[
\begin{aligned}
 H_{-p}C_p^i
 &=\sum_{v:\,\beta_v\in J_h}
       \nu(T^{h^*_{i\beta_v}}_{-p+1})H_{-p+1}\\
 &=\sum_{v:\,\beta_v\in J_h}
       F_p^{\beta_v\beta_v}C_p^iH_{-p+1}
  =q_{-p}C_p^iH_{-p+1}.
\end{aligned}
\]
Terms with $\beta_v\notin J_h$ vanish by
\eqref{eq:hyper-recursive-support} and the definition of
$H^-_{g,p}$ for the hyperedge $g$ containing $\beta_v$.
Thus
\[
 H_{-p}=q_{-p}\diag_a(H_{-p+1}).
\]
The induction hypothesis and \eqref{eq:hyper-q-diagonal}
now imply
\[
 H_p=q_pD_p=q_p,\qquad
 H_{-p}=q_{-p}D_{-p}=q_{-p}.
\]
This verifies \eqref{eq:general-normalization1}.

\emph{Relation \eqref{eq:general-normalization2}.}
The proof that \eqref{eq:general-normalization2} is preserved is very similar to the proof that \eqref{eq:general-normalization1} is preserved and therefore is omitted.

All defining relations of $R$ are preserved. Thus
$\nu:R\to S$ is a well-defined algebra homomorphism. Since $\nu(T^v_0)=v$ for all $v\in H^0$, $\nu$ is an $A$-algebra homomorphism.

\emph{The composites.}
Let $\alpha=\beta^{-1}:R\to L_0$ be the isomorphism constructed
in the proof of Theorem~\ref{thm:homogeneous-presentation}.
We first show that $\theta\nu=\alpha$.
The frame formulas \eqref{eq:hyper-frames} and the definition
\eqref{eq:hyper-theta} imply, for $p\geq1$,
\[
\begin{aligned}
 \theta\nu(T^v_p)&=\theta(E_p^{\alpha_v\alpha_v})
   =V_pvU_p=\alpha(T^v_p),\\
 \theta\nu(T^v_{-p})&=\theta(F_p^{\beta_v\beta_v})
   =V_{-p}vU_{-p}=\alpha(T^v_{-p}).
\end{aligned}
\]
At level zero, $\theta\nu(T^v_0)=v=\alpha(T^v_0)$.
Summing over $v$ also proves
$\theta\nu(Q_t)=V_tU_t=\alpha(Q_t)$ for every $t\in\Z$.
For the transition matrices, \eqref{eq:hyper-theta},
\eqref{eq:hyper-transitions} and \eqref{eq:hyper-frames} imply
\[
\begin{aligned}
 \theta\nu(T^{h_{ij}}_{p-1})
 &=\theta(E_p^{\alpha_{s(h)_i},i})B_p^j
   =V_ph_{ij}U_{p-1}=\alpha(T^{h_{ij}}_{p-1}),\\
 \theta\nu(T^{h^*_{ij}}_p)
 &=(B_p^j)^{\mathsf T}\theta(E_p^{i,\alpha_{s(h)_i}})
   =V_{p-1}h^*_{ij}U_p=\alpha(T^{h^*_{ij}}_p),\\
 \theta\nu(T^{h_{ij}}_{-p})
 &=(C_p^i)^{\mathsf T}\theta(F_p^{j,\beta_{r(h)_j}})
   =V_{-p+1}h_{ij}U_{-p}=\alpha(T^{h_{ij}}_{-p}),\\
 \theta\nu(T^{h^*_{ij}}_{-p+1})
 &=\theta(F_p^{\beta_{r(h)_j},j})C_p^i
   =V_{-p}h^*_{ij}U_{-p+1}=\alpha(T^{h^*_{ij}}_{-p+1}).
\end{aligned}
\]
For example, the first line uses
$\diag_b(V_{p-1})Y_{\bullet\alpha_{s(h)_i}}=V_ps(h)_i$ and
$X_{i\bullet}\diag_b(U_{p-1})B_p^j=h_{ij}U_{p-1}$.
It follows that $\theta\nu=\alpha$ and hence
$\mu\nu=\beta\theta\nu=\beta\alpha=\id_R$.

To obtain the other composite, let $i\in I_h$ and $k\in I_g$
with $s(h)_i=s(g)_k=v$. For $j\in J_h$ and $\ell\in J_g$,
\eqref{eq:hyper-transitions} and \eqref{eq:hyper-positive} imply
\[
 \nu(T^{h^*_{ij}}_pT^{g_{k\ell}}_{p-1})
 =(B_p^j)^{\mathsf T}E_p^{i,\alpha_v}E_p^{\alpha_v,k}B_p^\ell
 =(B_p^j)^{\mathsf T}E_p^{ik}B_p^\ell.
\]
Similarly, if $j\in J_h$ and $k\in J_g$ with $r(h)_j=r(g)_k=v$,
then for $i\in I_h$ and $\ell\in I_g$,
\eqref{eq:hyper-negative} implies
\[
 \nu(T^{h_{ij}}_{-p}T^{g^*_{\ell k}}_{-p+1})
 =(C_p^i)^{\mathsf T}F_p^{j,\beta_v}F_p^{\beta_v,k}C_p^\ell
 =(C_p^i)^{\mathsf T}F_p^{jk}C_p^\ell.
\]
As the block indices vary, the right-hand sides comprise all
possibly nonzero blocks of the generating matrices $E_p^{ik}$
and $F_p^{jk}$; their other blocks are zero by
\eqref{eq:hyper-recursive-support}. Since $\nu(T^v_0)=v$,
every generator of $S$ lies in $\nu(R)$, so $\nu$ is surjective.
For every $s=\nu(r)\in S$, we have
$\nu\mu(s)=\nu\mu\nu(r)=\nu(r)=s$.
Hence $\nu\mu=\id_S$, and $\mu$ and $\nu$ are inverse to each other.
\end{proof}

\section*{Declaration of generative AI and AI-assisted technologies in the manuscript preparation process}

During the preparation of this work, the author used ChatGPT to discuss mathematical ideas and check proofs. The author reviewed and edited the output as needed and takes full responsibility for the content of the published article.


\begin{thebibliography}{99}

\bibitem {AA} G. Abrams, G. Aranda Pino, \emph{The Leavitt path algebra of a graph}, J. Algebra {\bf 293} (2005), 319--334.

\bibitem {AH} G. Abrams, R. Hazrat, \emph{Connections between abelian sandpile models and the K-theory of weighted Leavitt path algebras}, Eur. J. Math., {\bf 9} (2023), 21.

\bibitem{aragoodearl} P. Ara, K. Goodearl, \emph{Leavitt path algebras of separated graphs}, J. reine angew. Math. {\bf 669} (2012), 165--224.

\bibitem {AMP} P. Ara, M.A. Moreno, E. Pardo, \emph{Nonstable K-theory for Graph Algebras}, Algebr. Represent. Theory {\bf 10} (2007) 157--178.

\bibitem{Bergman}
G.M.~Bergman,
Coproducts and some universal ring constructions,
\emph{Trans. Amer. Math. Soc.} \textbf{200} (1974), 33--88.

\bibitem {H-1} R. Hazrat, \emph{The graded structure of Leavitt path algebras}, Israel
J. Math {\bf 195} (2013), 833--895.

\bibitem{Hazrat}
R.~Hazrat,
\emph{Graded Rings and Graded Grothendieck Groups},
London Mathematical Society Lecture Note Series \textbf{435}, Cambridge University Press, Cambridge, 2016.

\bibitem{vitt56} W.G. Leavitt, \emph{Modules over rings of words}, Proc. Amer. Math. Soc. {\bf 7} (1956), 188--193. 

\bibitem{vitt57} W.G. Leavitt, \emph{Modules without invariant basis number,} Proc. Amer. Math. Soc. {\bf 8} (1957), 322--328.

\bibitem{vitt62} W.G.  Leavitt, \emph{The module type of a ring,} Trans. Amer. Math. Soc. {\bf 103} (1962) 113--130. 

\bibitem{vitt65} W.G. Leavitt, \emph{The module type of homomorphic images,} Duke Math. J. {\bf 32} (1965), 305--311.

\bibitem{Hyper}
R.~Preusser,
Leavitt path algebras of hypergraphs,
\emph{Bull. Braz. Math. Soc. (N.S.)} \textbf{51} (2020), 185--221.

\bibitem{Survey}
R.~Preusser,
Weighted Leavitt path algebras---an overview,
\emph{Zap. Nauchn. Sem. POMI} \textbf{531} (2024), 157--237.
\end{thebibliography}
\end{document}